\documentclass[a4paper,12pt,reqno]{amsart}
\title{The probability that a random triple of dice is transitive}
\author{D. H. J. Polymath}

\usepackage{amsmath, wrapfig}
\usepackage{dsfont, a4wide, amsthm, amssymb, amsfonts, graphicx}
\usepackage{fancyhdr, xspace, psfrag, setspace, supertabular, color}
\usepackage{hyperref}
\usepackage{breakurl} 
\usepackage{bbm}
\usepackage{stmaryrd}
\usepackage{txfonts}

\usepackage[T1]{fontenc}

\newtheorem{theorem}{Theorem}[section]

\newtheorem{lemma}[theorem]{Lemma}
\newtheorem{claim}[theorem]{Claim}
\newtheorem{corollary}[theorem]{Corollary}
\newtheorem{conjecture}[theorem]{Conjecture}

\def\e{\epsilon}
\def\E{\mathbb{E}}
\def\Z{\mathbb{Z}}
\def\R{\mathbb{R}}
\def\T{\mathbb{T}}

\def\P{\mathbb{P}}

\def\a{\alpha}
\def\be{\beta}

\def\b1{\mathbbm{1}}

\def\ol{\overline}

\def\var{\mathop{\rm var}}

\begin{document}
\begin{abstract}
An $n$-\emph{sided die} is an $n$-tuple of positive integers. We say that a die $(a_1,\dots,a_n)$ \emph{beats} a die $(b_1,\dots,b_n)$ if the number of pairs $(i,j)$ such that $a_i>b_j$ is greater than the number of pairs $(i,j)$ such that $a_i<b_j$. We show that for a natural model of random $n$-sided dice, if $A, B$ and $C$ are three random dice then the probability that $A$ beats $C$ given that $A$ beats $B$ and $B$ beats $C$ is approximately 1/2. In other words, the information that $A$ beats $B$ and $B$ beats $C$ has almost no effect on the probability that $A$ beats $C$. This proves a statement that was conjectured by Conrey, Gabbard, Grant, Liu and Morrison for a different model.
\end{abstract}
\maketitle

\section{Introduction}

It is an amusing fact, first observed by Bradley Efron in the 1960s, that there can be four dice $A_0,A_1,A_2,A_3$, each with six sides but with non-standard numberings such that if they are all rolled, then each of the four events ``$A_i$ shows a higher number than $A_{i+1}$" occurs with probability 2/3 (where $i+1$ is interpreted mod 4). Thus, if we say that one die \emph{beats} another if it has a better-than-50\% chance of showing a higher number, then the relation ``beats" is not transitive.

Much more recently, Conrey, Gabbard, Grant, Liu and Morrison \cite{CGG}
decided to investigate how common a phenomenon intransitivity is. They defined a notion of random $n$-sided dice, defined a suitable relation ``beats" for such dice, and did some computer experiments that indicated, to their surprise, that if $A, B$ and $C$ are three random dice, then the probability that $A$ beats $C$ given that $A$ beats $B$ and $B$ beats $C$ is, for large $n$, approximately equal to 1/2. That is, the information that $A$ beats $B$ and $B$ beats $C$ gives almost no clue about whether $A$ beats $C$. 

The definition they gave of a random $n$-sided die, which we shall refer to as the \emph{multiset model}, is as follows: they define an $n$-\emph{sided die} to be a multiset with $n$ elements that add up to $n(n+1)/2$ (or equivalently average $(n+1)/2$), and a random $n$-sided die is simply an $n$-sided die chosen uniformly at random. An equivalent definition is that a random $n$-sided die is a random non-decreasing sequence $(a_1,\dots,a_n)$ of positive integers between 1 and $n$ that add up to $n(n+1)/2$. For example, the 4-sided dice are $(1,1,4,4), (1,2,3,4), (1,3,3,3),(2,2,2,4)$, and $(2,2,3,3)$. 

Given two random $n$-sided dice $A=(a_1,\dots,a_n)$ and $B=(b_1,\dots,b_n)$, we say that $A$ \emph{beats} $B$ if the number of pairs $(i,j)$ such that $a_i>b_j$ is greater than the number of pairs $(i,j)$ such that $a_i<b_j$. For example, the die $A=(1,1,4,4)$ beats the die $B=(1,3,3,3)$ because there are eight pairs $(i,j)$ with $a_i>b_j$ and only six with $a_i<b_j$. If the two numbers are equal we say that $A$ \emph{ties with} $B$.

Conrey, Gabbard, Grant, Liu and Morrison \cite{CGG} made the following two conjectures.

\begin{conjecture} \label{ties}
Let $n$ be a positive integer and let $A$ and $B$ be independent random $n$-sided dice in the multiset model. Then the probability that $A$ ties with $B$ is $o(1)$.
\end{conjecture}

\begin{conjecture} \label{transitive}
Let $n$ be a positive integer and let $A,B$ and $C$ be independent random $n$-sided dice in the multiset model. Then the probability that $A$ beats $C$ given that $A$ beats $B$ and $B$ beats $C$ is $\frac 12+o(1)$. 
\end{conjecture}

They also conjectured a strengthening of Conjecture \ref{transitive}, which is the following. Recall that a \emph{tournament} is a complete graph for which every edge is given a direction. We shall regard it as a set $T$ of ordered pairs of distinct elements of a set $V$ such that for any two distinct elements $v,w$ of $V$ exactly one of $(v,w)$ or $(w,v)$ belongs to $T$.

\begin{conjecture} \label{quasirandom}
Let $T$ be a tournament with vertices $1,2,\dots,k$. Then if $A_1,\dots,A_k$ are independent random $n$-sided dice in the multiset model, then the probability that for each $1\leq i<j\leq k$ we have that $A_i$ beats $A_j$ if and only if $(i,j)\in T$ is $2^{-\binom k2}+o(1)$.
\end{conjecture}

The conclusion about the tournament $T$ is stating that it is \emph{quasirandom} in a sense introduced by Chung and Graham \cite{CG}. It turns out to be equivalent to the statement that for all but a fraction $o(1)$ of the pairs of vertices $x,y$, the fraction of vertices $z$ such that either $(x,z)$ and $(y,z)$ belong to $T$ or $(z,x)$ and $(z,y)$ belong to $T$ is $\frac 12+o(1)$. There are many different equivalent conditions for quasirandomness: the one conjectured to hold by Conrey, Gabbard, Grant, Liu and Morrison states that all small tournaments occur in $T$ with approximately the frequency one would expect in a random tournament.

Conrey, Gabbard, Grant, Liu and Morrison also looked at other models, and the experimental evidence was surprisingly sensitive to the model chosen, with Conjecture \ref{transitive} (and hence also Conjecture \ref{quasirandom}) appearing to be false for most of them. However, there was one other model for which it seemed to be true, which we shall refer to as the \emph{balanced sequences model}. Here an $n$-sided die is simply a sequence $(a_1,\dots,a_n)$ of elements of $\{1,2,\dots,n\}$ that adds up to $n(n+1)/2$ and a random $n$-sided die is an $n$-sided die chosen uniformly at random. Note that permuting a sequence does not affect which other sequences it beats. If we say that two dice that are permutations of one another are \emph{equivalent}, then the difference between the balanced sequences model and the multiset model is that the multiset model gives the same weight to each equivalence class, while the balanced sequences model gives the same weight to each individual sequence.

The main results of this paper are that Conjectures \ref{ties} and \ref{transitive} are true for the balanced sequences model.

\begin{theorem} \label{ties-thm}
Let $n$ be a positive integer and let $A$ and $B$
be independent random $n$-sided dice in the balanced sequences model. 
Then the probability that $A$ ties with $B$ is $o(1)$.
\end{theorem}

\begin{theorem} \label{transitive-thm}
Let $n$ be a positive integer and let $A$, $B$ and $C$ be independent
random $n$-sided dice in the balanced sequences model. Then the probability
that $A$ beats $C$ given that $A$ beats $B$ and $B$ beats $C$
is $\frac{1}{2}+o(1)$.
\end{theorem}

The method of proof can be summarized as follows. We begin by showing that Theorem~\ref{transitive-thm}
is equivalent to the statement that almost every random die beats approximately half the other dice and is beaten by approximately half the other dice (and therefore ties with almost no dice). We then argue that unless a die $A$ has a very ``atypical" distribution, then it is indeed the case that it beats approximately half the other dice and is beaten by approximately half the other dice. We can regard this last statement as the claim that if $(b_1,\dots,b_n)$ is a random sequence of elements of $\{1,2,\dots,n\}$, then the probability that it is beaten by $A$ given that it sums to $n(n+1)/2$ is approximately 1/2, as is the probability that it beats $A$ given the same condition. It turns out that this is true provided that a certain sum of independent random variables with values in $\Z$ is sufficiently close to a discrete Gaussian distribution. In order to prove this, we need a rather explicit quantitative local central limit theorem, which we prove by standard Fourier-analytic means, using probabilistic arguments to prove that the behaviour we need the Fourier transform (or characteristic function) to satisfy holds for almost all dice $A$.

This paper is the result of an open online collaboration between several authors. A complete record of the discussion that led to its existence can be found in a series of seven consecutive blog posts and comments on them, of which the first is \url{https://gowers.wordpress.com/2017/04/28/a-potential-new-polymath-project-intransitive-dice/}. The posts belong to a category entitled polymath13.

\subsection{Subsequent work} A work \cite{fourdice} subsequent to the initial
draft of this paper has disproved
Conjecture~\ref{quasirandom} in the balanced sequences model where the faces
are uniform real numbers in $[0,1]$ (or, equivalently, in $[0,n]$). A couple of
arguments from \cite{fourdice} have been incorporated into the proofs
in this paper.


\section{The preliminary reduction}\label{reduction}

We begin with a lemma about tournaments, or rather about near tournaments, by which we mean subgraphs of tournaments with $n$ vertices and $(1-o(1))\binom n2$ edges. Given a triple of vertices $(x,y,z)$, we shall call it \emph{intransitive} if the subgraph induced by the three vertices is a directed cycle of length 3, and \emph{transitive} if it is a triangle but not a directed 3-cycle. The \emph{out-degree} $d_+(x)$ of a vertex $x$ is the number of vertices $y$ such that $(x,y)$ is an edge, and the \emph{in-degree} $d_-(x)$ is the number of vertices $y$ such that $(y,x)$ is an edge.

\begin{lemma} \label{equivalent}
Let $T$ be a subgraph of a tournament with $n$ vertices and $(1-o(1))\binom n2$ edges. Then the following two statements are equivalent.
\begin{enumerate}
\item The probability that a random triple of vertices is intransitive is $\frac 14+o(1)$.
\item If $x$ is a random vertex, then $d_+(x)=(\frac 12+o(1))n$ with probability $1-o(1)$.
\end{enumerate}
\end{lemma} 

We shall apply this lemma to the tournament of three randomly chosen dice.
Note that, assuming Theorem~\ref{ties-thm}, if the first statement
of the lemma holds for this tournament, then Theorem~\ref{transitive-thm} follows.

\begin{proof}
Write $x\to y$ if $(x,y)$ is an edge of $T$. First let us count the number of triples $(x,y,z)$ such that $x\to y\to z$. A directed triangle $xyz$ in $T$ gives rise to three such triples, namely $(x,y,z), (y,z,x)$ and $(z,x,y)$. Any other triangle gives rise to just one: for example, if $x\to y$, $x\to z$ and $y\to z$, then the only triple we obtain is $(x,y,z)$. Since the number of triangles is $(1-o(1))\binom n3$, we find that the number of triples $(x,y,z)$ such that $x\to y\to z$ is $(1-o(1))\binom n3$ plus twice the number of directed triangles. Note that $\binom n3=(1+o(1))n^3/6$.

But the number of such triples is also $\sum_yd_+(y)d_-(y)$. Since the number of edges is $(1-o(1))\binom n2$, this is equal to $\left(\sum_yd_+(y)(n-d_+(y))\right)+o(n^3)$. Also, $\sum_yd_+(y)=(1-o(1))\binom n2=(1-o(1))n^2/2$, so this is $n^3/2+o(n^3)-\sum_yd_+(y)^2$. Therefore, twice the number of directed triangles is $n^3/3+o(n^3)-\sum_yd_+(y)^2$.

If a random triple of vertices has a probability $\frac 14+o(1)$ of being intransitive, then twice the number of directed triangles is also $(\frac 12+o(1))\binom n3=(\frac 1{12}+o(1))n^3$. It follows that $\sum_yd_+(y)^2=(\frac 14+o(1))n^3$, and therefore that $\E_yd_+(y)^2=(\frac 14+o(1))n^2$. But $\E_yd_+(y)=(\frac 12+o(1))n$, so $\var(d_+(y))=o(n^2)$, which implies that $d_+(y)=(\frac 12+o(1))n$ with probability $1-o(1)$.

The steps in the previous paragraph can also be reversed, so the lemma is proved.
\end{proof}

\section{
\texorpdfstring{A random variable related to an $n$-sided die and a second reduction}
{A random variable related to an n-sided die and a second reduction}
} \label{rvs}

Write $[n]$ for the set $\{1,2,\dots,n\}$. Given an $n$-sided die $A=(a_1,\dots,a_n)$ (in fact the definition we are about to give applies to any sequence in $[n]^n$) we define a cumulative distribution function $f_A$ by 
\[f_A(j)=|\{i\in[n]:a_i<j\}|+\frac 12|\{i\in[n]:a_i=j\}|.\]
We typically expect $f_A(j)$ to be around $j-\frac 12$, so it is convenient also to define a function $g_A$ by $g_A(j)=f_A(j)-j+\frac 12$. For a fixed $A$, we shall be interested in the random variable $(g_A(j), j-(n+1)/2)$, which is defined on $[n]$. More precisely, we choose $j$ uniformly from $[n]$ and evaluate the pair $(g_A(j),j-(n+1)/2)$. 

To see why this is useful to look at, let us do a few simple calculations. 

First of all, 
\begin{align*}\sum_jf_A(j)&=\sum_j\sum_i\Bigl(\b1_{[a_i<j]}+\frac 12\b1_{[a_i=j]}\Bigr)\\
&=\sum_i\sum_j\Bigl(\b1_{[a_i<j]}+\frac 12\b1_{[a_i=j]}\Bigr)\\
&=\sum_i(n-a_i+1/2)\\
&=n^2/2,
\end{align*}
where the last equality follows from the fact that $A$ is an $n$-sided die and therefore $\sum_ia_i=n(n+1)/2$. This gives us that 
\[\sum_jg_A(j)=n^2/2-\sum_j(j-1/2)=n^2/2-n(n+1)/2+n/2=0,\] 
and therefore that the mean of the random variable $(g_A(j),j-(n+1)/2)$ is $(0,0)$.

Next, let $B=(b_1,\dots,b_n)$ be another $n$-sided die. Then 
\[\sum_jf_A(b_j)=\sum_j\sum_i\Bigl(\b1_{[a_i<b_j]}+\frac 12\b1_{[a_i=b_j]}\Bigr)=\bigl|\{(i,j):a_i<b_j\}\bigr|+\frac 12\bigl|\{(i,j):a_i=b_j\}\bigr|.\]
But 
\[\sum_jg_A(b_j)=\sum_j(f_A(b_j)-b_j+1/2)=\sum_jf_A(b_j)-n^2/2,\]
where the last inequality follows from the fact that $\sum_jb_j=n(n+1)/2$. It follows that
\[\sum_jg_A(b_j)=\bigl|\{(i,j):a_i<b_j\}\bigr|+\frac 12\bigl|\{(i,j):a_i=b_j\}\bigr|-n^2/2.\]
Since there are $n^2$ pairs $(i,j)$, this tells us that $\sum_jg_A(b_j)>0$ if and only if
\[\bigl|\{(i,j):a_i<b_j\}\bigr|+\frac 12\bigl|\{(i,j):a_i=b_j\}\bigr|>\bigl|\{(i,j):a_i>b_j\}\bigr|+\frac 12\bigl|\{(i,j):a_i=b_j\}\bigr|,\]
which is true if and only if $B$ beats $A$. Similarly $A$ beats $B$ if and only if $\sum_jg_A(b_j)<0$.

We will therefore be done with Theorems~\ref{ties-thm} and~\ref{transitive-thm} if we can prove the following claim.

\begin{claim}\label{main}
If $A$ is a random $n$-sided die, then with probability $1-o(1)$ we have that the proportion of $n$-sided dice $B=(b_1,\dots,b_n)$ with $\sum_jg_A(b_j)>0$ is $\frac 12+o(1)$.
\end{claim}

The proof that this claim is sufficient requires one small observation. Given a die $A=(a_1,\dots,a_n)$, define the \emph{complementary die} $\ol A$ to be the sequence $(n+1-a_1,\dots,n+1-a_n)$. Then $A$ beats $B$ if and only if $\ol B$ beats $\ol A$. So if the claim is true, then with probability $1-o(1)$, the proportion of $B$ such that $B$ beats $A$ is $\frac 12+o(1)$ and the proportion of $B$ such that $\ol B$ beats $\ol A$ is also $\frac 12+o(1)$, which implies that the proportion of $B$ such that $A$ ties with $B$ is $o(1)$. 

\section{A heuristic argument for Claim \ref{main}}

We begin by explaining why one would expect Claim \ref{main} to be true. Once we have done that, we shall turn our heuristic argument into a rigorous one. It is at that point that we shall need to prove a local central limit theorem with sufficiently explicit bounds.

Let $(b_1,\dots,b_n)$ be a purely random sequence belonging to $[n]^n$ -- that is, one where the $b_i$ are chosen uniformly and independently from $[n]$ and there is no restriction on the sum. Then to prove Claim \ref{main} for a fixed $A$ we need to show that
\[\P\Bigl[\sum_jg_A(b_j)>0\Bigm|\sum_jb_j=n(n+1)/2\Bigr]=\frac 12+o(1),\]
which is equivalent to the assertion that
\[\P\Bigl[\sum_jg_A(b_j)>0\Bigm|\sum_j(b_j-(n+1)/2)=0\Bigr]=\frac 12+o(1).\]
But $b_1,\dots,b_n$ are uniformly and independently chosen from $[n]$. Thus, if we write $(X_j,Y_j)$ for the random variable $(g_A(b_j),b_j-(n+1)/2)$, then $(X_1,Y_1),\dots,(X_n,Y_n)$ are $n$ independent copies of the random variable $(g_A(j),j-(n+1)/2)$ mentioned earlier, and we are concerned with the sum $\sum_{j=1}^n(X_j,Y_j)$, which we shall write as $(X,Y)$. 

The central limit theorem suggests that the distribution of this sum will be approximately Gaussian, and since each $(X_i,Y_i)$ has mean $(0,0)$ we would in particular expect that the distribution would be approximately symmetric about the origin. Also, we would expect a typical value of $g_A(j)$ to have magnitude around $\sqrt n$, so the standard deviation of $X$ ought to be around $n$. Also $Y$ has standard deviation of order $n^{3/2}$ and the two random variables, though correlated, will probably not be too heavily correlated.

If all these heuristics are correct, then the probability that $X=0$ given that $Y=0$ should be of order $n^{-1}$, and certainly $o(1)$. The symmetry should imply that $\P[X>0|Y=0]\approx\P[X<0|Y=0]$, and these statements taken together would give us that $\P[X>0|Y=0]=\frac 12+o(1)$ and $\P[X<0|Y=0]=\frac 12+o(1)$, which is equivalent, as we have seen, to the statement that the proportion of dice that beat $A$ is $\frac 12+o(1)$ and the proportion of dice that $A$ beats is $\frac 12+o(1)$. 

The reason this heuristic argument cannot immediately be turned into a proof is that the central limit theorem is too blunt a tool. There are two reasons for this. The first is that although it tells us that a sum of i.i.d.~random variables will converge to a Gaussian, it does not tell us how fast that convergence will occur, and we need it to have occurred (to within a small error) when we take a sum of $n$ copies of $(g_A(j),j)$. And we cannot just let $n$ tend to infinity because the random variables themselves depend on $n$. This second problem applies not just to the central limit theorem but also to the Berry--Esseen theorem, which gives a rate of convergence in the central limit theorem, but with a constant that (necessarily) depends on the random variable.

A second problem is that the notion of convergence in the central limit theorem and the Berry--Esseen theorem is not suitable for our purposes. We need to be able to estimate the probability that $(X,Y)$ belongs to the positive x-axis, which is a ``probability zero event" from the point of view of the central limit theorem and Berry--Esseen theorem. Instead, we need a \emph{local central limit theorem}, the name given to versions of the central limit theorem that can give us estimates for the density function at individual values. Unfortunately, the local central limit theorems that appear in the literature tend still to involve inexplicit constants that depend on the random variable, again necessarily. (We did find an exception to this in~\cite{LCLT}, but it proved a one-dimensional theorem where we need a two-dimensional one.) 

In the end, we have proved for ourselves a local central limit theorem that is tailored to our application. It is not hard to prove using Fourier analysis, which is one of the standard methods for proving such results, but it requires the random variable to have certain properties, as we shall explain later, in order for us to be able to make the implied constant explicit. So the rest of the proof splits into two parts: first we shall prove that the random variable $(U,V)=(g_A(j),j-(n+1)/2)$ has certain properties with high probability (when $A$ is a random $n$-sided die). Then we shall use those properties to establish a suitable local central limit theorem, after which the argument will essentially be finished. 

\section{
\texorpdfstring{Properties of the random variable $(U,V)$}
{Properties of the random variable (U,V)}
}

From now on we shall write $(U,V)$ for an instance of the random variable
$(g_A(j),j)$ for a fixed die $A$.
In this section we shall obtain an upper bound for the essential supremum of $U$
and an upper bound on the size of the characteristic function of $(U,V)$. All these bounds will hold with probability $1-o(1)$ when $A$ is a random $n$-sided die in the balanced sequences model.

\subsection{
\texorpdfstring{An upper bound for $\|U\|_\infty$.}
{An upper bound for U\textunderscore infinity norm.}
}

We begin with an almost standard fact (Lemma \ref{probzero} below), but for convenience we provide a complete proof. (The fact and its proof could be thought of as a weakening of a very special case of a one-dimensional local central limit theorem.) First we prove an even more basic lemma.

\begin{lemma}
Let $I_n$ be the set $\{-(n-1)/2,-(n-3)/2,\dots,(n-3)/2,(n-1)/2\}$ and let $f$ be defined on $\frac 12\Z$ by taking $f(x)=n^{-1}$ if $x\in I_n$ and $f(x)=0$ otherwise. (Thus, $f(x)=\P[V=x]$.) Then the $k$-fold convolution $f^{*k}$ of $f$ is supported on $\Z$ except if $k$ is odd and $n$ is even, in which case it is supported on $\Z+\frac 12$. In all cases, $f^{*k}$ is an even function, and its non-zero values increase when $x<0$ and decrease when $x>0$.
\end{lemma}

\begin{proof}
The statements about the support and the symmetry are trivial. To prove the increasing and decreasing properties, we note that they follow easily by induction. Indeed, let $g$ be any even function supported on an arithmetic progression of common difference 1 that increases towards the middle, and let $x\geq 0$. Then the inner product of $g$ with $I_n+x$ is greater than or equal to the inner product of $g$ with $I_n+x+1$, since $g(x-(n-1)/2)\geq g(x+1+(n-1)/2)$. Therefore, $g*f$ decreases when $x$ is positive, and by symmetry it increases when $x$ is negative (when we restrict to appropriate supports).  
\end{proof}

What we care about here is that when $k=n$, the maximum of $f^{*n}$ is attained at zero. And all we really need from the next lemma is that the probability that $Y=0$ is not tiny.

\begin{lemma} \label{probzero}
Let $(a_1,\dots,a_n)$ be an element of $[n]^n$ chosen uniformly at random. Then the probability that $\sum_ia_i=n(n+1)/2$ is at least $n^{-3/2}/4$.
\end{lemma}

\begin{proof}
The probability that $\sum_ia_i=n(n+1)/2$ is $f^{*n}(0)$. Equivalently, it is the probability that $Y=0$, where $Y$ is the sum of $n$ independent copies of $V$. The variance of $V$ is at most $n^2/4$, so the variance of $Y$ is at most $n^3/4$. Therefore, by Chebyshev's inequality, the probability that $|Y|\geq n^{3/2}$ is at most $1/4$, which implies that the probability that $|Y|\leq n^{3/2}$ is at least $3/4$. Since $0$ is the most likely value of $Y$, it follows that $Y=0$ with probability at least $\frac{3}{4(2n^{3/2}+1)}\geq n^{-3/2}/4$.
\end{proof}

We shall now obtain an upper bound for $\|U\|_\infty$. Our method is to obtain an upper bound that holds with such high probability for a purely random element of $[n]^n$ that it continues to hold with high probability even when we condition on the sum being $n(n+1)/2$.

\begin{lemma} \label{maxnorm}
Let $A$ be a random $n$-sided die. Then with probability $1-8n^{-11/2}$ we have that $\max_j|g_A(j)|\leq 2\sqrt{n\log n}$.
\end{lemma}

\begin{proof}
Let $(a_1,\dots,a_n)$ be a purely random sequence -- that is, an element of $[n]^n$ chosen uniformly at random. For each $j$, let $n_A(j)$ be the number of $i$ such that $a_i\leq j$. Note that $f_A(j)$ is the average of $n_A(j-1)$ and $n_A(j)$. 

Now $n_A(j)$ is a sum of $n$ independent Bernoulli random variables of mean $j/n$. By Chernoff's bounds, the probability that $|n_A(j)-j|\geq m$ is at most $2\exp(-2m^2/n)$. Therefore, the probability that there exists $j$ such that $|n_A(j)-j|\geq m$ is at most $2n\exp(-2m^2/n)$. Setting $m=2\sqrt{n\log n}$, this is at most $2n\exp(-8\log n)=2n^{-7}$. By Lemma \ref{probzero}, if we now condition on the event that $\sum_ia_i=n(n+1)/2$, then this probability rises to at most $8n^{-11/2}$.

If no such $j$ exists, then for every $j$ we have that
\[\big|(n_A(j-1)+n_A(j))/2 - (j-1+j)/2\big|\leq 2\sqrt{n\log n},\]
by the triangle inequality. The left-hand side of this inequality is $|g_A(j)|$.
\end{proof}

\subsection{Tools for bounding the characteristic function}

In this section we gather a few tools that will be useful for bounding
the characteristic function of $(U,V)$. We start by stating Azuma's inequality.

\begin{lemma}
\label{deviation}
Let the random variables $Y_0,\ldots,Y_n$ form a bounded difference
submartingale 
adapted to a filtration $\mathcal{F}_0,\ldots,\mathcal{F}_n$, 
that is, the random variable $Y_i$ is measurable with respect
to the $\sigma$-algebra $\mathcal{F}_i$ for each $i$, and almost surely
$\E\left[Y_{i+1}-Y_i\mid \mathcal{F}_i\right]\ge 0$ and
$|Y_{i+1}-Y_i|\le c$. Then, for every $n$ and every $\e>0$,
\[
\P[Y_n\le Y_0-\e]\le\exp\left(
-\frac{\e^2}{2cn}
\right)\;.
\]
\end{lemma}

In fact, we need a small modification of Lemma \ref{deviation}. Roughly speaking, Lemma \ref{deviation} works fine except if certain events of very low probability occur, so we now prove a variant of the lemma that shows that low-probability events do not mess up the conclusion by too much.

\begin{lemma} \label{deviation2}
Let $Y_0,\ldots,Y_n$ be random variables 
adapted to a filtration $\mathcal{F}_0,\ldots,\mathcal{F}_n$
and with bounded differences $|Y_{i+1}-Y_i|\le 1$.
Furthermore, for each $i$ let $D_{i}$ denote the event
\[
\E[Y_{i}-Y_{i-1}\mid\mathcal{F}_{i-1}]< 0.
\] 
Then for every $\e>0$,
\[
\P[Y_n\le Y_0-\e]\le
\exp\left(-\frac{\e^2}{2n}\right)+\Pr\left[\bigcup_{i=1}^n D_i\right]\;.
\]
\end{lemma}

\begin{proof}
The only difference between this statement and Lemma~\ref{deviation} is that the random
variables $Y_0,\ldots,Y_n$ do not always satisfy the
submartingale property.
To fix this, let $Y'_0,\ldots,Y'_n$ be a sequence of random variables given by
$Y'_0=Y_0$ and for $i\geq 1$,
\[Y'_{i}=\begin{cases}
Y_{i}&\text{if none of events $D_1,\ldots,D_{i}$ happen,}\\
Y'_{i-1}&\text{otherwise.}
\end{cases}
\]
Keeping in mind the definitions of $D_i$ and $Y'_i$, by induction we see that $Y'_0,\ldots,Y'_n$ form a bounded difference submartingale with
respect to the filtration $\mathcal{F}_0,\ldots,\mathcal{F}_n$.
Consequently, by Lemma~\ref{deviation},
\[
\Pr[Y'_n\le Y'_0-\e]\le\exp\left(-\frac{\e^2}{2n}\right)\;. 
\]
On the other hand, if
none of $D_1,\ldots,D_{n}$ happened, then $Y'_n=Y_n$.
Hence,
\[
\Pr[Y_n\le Y_0-\e]\le\Pr[Y'_n\le Y'_0-\e]+\Pr[Y'_n\ne Y_n]
\le\exp\left(-\frac{\e^2}{2n}\right)+
\Pr\left[\bigcup_{i=1}^{n} D_i\right]\;.
\qedhere\]
\end{proof}

We shall also need to show that certain sums of independent random variables lie in certain ranges with probability bounded away from zero. The Berry--Esseen theorem is sufficient for this. (More elementary approaches are possible too, but a bit more cumbersome.) The version of the Berry--Esseen theorem we shall use is the following. 

\begin{theorem} \label{berryesseen}
Let $X_1,\dots,X_n$ be independent and identically distributed random variables with $\E X_i=0$, $\E X_i^2=\sigma^2$ and $\E|X_i|^3=\rho$ for each $i$. Let $X=X_1+\dots+X_n$, let $x$ be a real number, and let $Y$ be a random variable with the standard normal distribution. Then
\[\Bigl|\P[X\leq x\sigma\sqrt n]-\P[Y\leq x]\Bigr|\leq\frac\rho{2\sigma^3\sqrt n}.\]
\end{theorem}

\noindent Berry and Esseen obtained the same theorem but with a larger absolute constant on the right-hand side. The constant of 1/2 (in fact, slightly better) was obtained by Shevtsova \cite{shevtsova}.

\begin{corollary} \label{inrange}
Let $X_1,\dots,X_n$ and $X$ be as in the previous lemma. Then
\[\P[\sigma\sqrt n\leq X\leq2\sigma\sqrt n]\geq \frac 18-\frac\rho{\sigma^3\sqrt n}\]
and
\[\P[-2\sigma\sqrt n\leq X\leq-\sigma\sqrt n]\geq \frac 18-\frac\rho{\sigma^3\sqrt n}.\]
\end{corollary}

\begin{proof}
By Theorem \ref{berryesseen}, the first probability differs from $\P[1\leq Y\leq 2]$ by at most $\rho/\sigma^3\sqrt n$. But a standard normal lies in the interval $[1,2]$ with probability greater than 1/8, which gives the first estimate. The second is proved in the same way.
\end{proof}

Finally, we shall need a quantitative version of the Poisson
limit theorem for sums of independent random variables
with small expectations. For this version see~(1.1)
in~\cite{poisson}.
\begin{theorem}\label{poisson}
Let $X_1,\ldots,X_n$ be independent and
identically distributed Bernoulli random variables
with $\E X_i=\lambda/n$. Let $X=X_1+\ldots+X_n$,
let $S$ be a subset of integers and let $Y$
be a random variable with the Poisson distribution
with mean $\lambda$. Then
\[\Bigl|\P[X\in S]-\Pr[Y\in S]\Bigr|\le\lambda^2/n\;.
\]
In particular,
$\P[X=0]\ge e^{-\lambda}-\lambda^2/n$
and $\P[X=1]\ge \lambda e^{-\lambda}-\lambda^2/n$.
\end{theorem}

\subsection{Bounding the magnitude of the characteristic function away from 1}

As in many proofs of central-limit-type theorems, we shall use characteristic functions, or equivalently Fourier analysis. Recall that the characteristic function of $(U,V)$ is the function 
\[\hat f(\a,\be)=\E e(\a U+\be V),\]
where $e(x)$ is shorthand for $\exp(2\pi i x)$. We shall say more about the importance of the characteristic function later, but for now we simply ask the reader to take it on trust that it will be useful to us to bound $|\hat f(\a,\be)|$ away from 1, except when $\a$ and $\be$ are very small. (Note that when $\a=\be=0$ then $\hat f(\a,\be)=1$, so some dependence is necessary.)

We shall also make use of the following small technicality.

\begin{lemma}\label{quadruple}
Let $\theta_1,\theta_2,\theta_3$ and $\theta_4$ and $\e$ be real numbers such that the distance from $\theta_1-\theta_2-\theta_3+\theta_4$ to the nearest integer is at least $\e$. Then $|e(\theta_1)+e(\theta_2)+e(\theta_3)+e(\theta_4)|\leq 4-\e^2$.
\end{lemma}

\begin{proof}
First note that for any $\theta$ and $\phi$ with $|\theta-\phi|\leq 1/2$ we have the inequality
\[|e(\theta)+e(\phi)|^2=2+2\cos(2\pi(\theta-\phi))\leq 4-(2\pi)^2(\theta-\phi)^2+(2\pi)^4(\theta-\phi)^4/12\leq 4(1-(\theta-\phi)^2)^2,\]
and therefore 
\[|e(\theta)+e(\phi)|\leq 2(1-(\theta-\phi)^2).\]
Since adding an integer makes no difference, we may assume that $|\theta_1-\theta_2|$ and $|\theta_3-\theta_4|$ are both at most $1/2$. It follows that
\[|e(\theta_1)+e(\theta_2)+e(\theta_3)+e(\theta_4)|\leq 4-2(\theta_1-\theta_2)^2-2(\theta_3-\theta_4)^2\leq 4-(\theta_1-\theta_2-\theta_3+\theta_4)^2,\]
which proves the result, since $|\theta_1-\theta_2-\theta_3+\theta_4|\geq\e$.
\end{proof}

Since $(U,V)=(g_A(j),j)$ for a randomly chosen $j\in[n]$, it follows that
\[\hat f(\a,\be)=n^{-1}\sum_{j=1}^ne(\a g_A(j)+\be j).\]

Note that if $j_1-j_2=j_3-j_4$ and $\e\leq|\a(g_A(j_1)-g_A(j_2)-g_A(j_3)+g_A(j_4))|\leq 1/2$
for some $\e>0$, then for every $\be$ we have that
\[\frac{1}{4}|e(\a g_A(j_1)+\be j_1)+e(\a g_A(j_2)+\be j_2)+e(\a g_A(j_3)+\be j_3)+e(\a g_A(j_4)+\be j_4)|\leq 1-\e^2/4.\]
This follows by setting $\theta_i=\a g_A(j_i)+\be j_i$ and observing that since $\be(j_1-j_2-j_3+j_4)=0$, the distance from $\theta_1-\theta_2-\theta_3+\theta_4$ to the nearest integer is the same as the distance from $\a(g_A(j_1)-g_A(j_2)-g_A(j_3)+g_A(j_4))$ to the nearest integer, which by hypothesis is at least $\e$. This allows us to apply Lemma \ref{quadruple}.

Our strategy now will be to prove that with high probability there are many non-overlapping quadruples $(j_1,j_2,j_3,j_4)$ with $j_1-j_2=j_3-j_4$ and $\e\leq|\a(g_A(j_1)-g_A(j_2)-g_A(j_3)+g_A(j_4))|\leq 1/2$
for appropriately chosen $\e$.

\begin{lemma}\label{quadruple-clt}
For large enough $n$ the following holds. 
Let $j,m,\alpha$ be such that $10^6\leq m\leq n/\log n$,
$j+3m\le n/2$ and $|\alpha|\sqrt{m}\le 1/12$. Let $z=|\{h:a_h\le j+2m\}|$ and assume that
$z\le 2n/3$.

Then, if $A$ is a random sequence in $[n]^n$, after conditioning on 
the value of $z$,
with probability at least $1/10$,
the distance from
$\alpha(g_A(j)-g_A(j+m)-g_A(j+2m)+g_A(j+3m))$
to the nearest integer
is at least $|\alpha|\sqrt{m/6}$.
\end{lemma}

\begin{proof}
After conditioning on $z$, we have
\[f_A(j+3m)=z+\big|\{h:j+2m<a_h<j+3m\}\big|+\frac 12\big|\{h:a_h=j+3m\}\big|,\]
which is a sum of $n-z$ independent random variables $X_1,\dots,X_{n-z}$, each of which takes the value $1$ with probability $(m-1)/(n-j-2m)$, $1/2$ with probability $1/(n-j-2m)$, and $0$ with probability $(n-j-3m)/(n-j-2m)$.

We shall now apply Corollary \ref{inrange} to these random variables. Back-of-envelope calculations give the bounds $\sigma^2,\rho\leq 4m/n$ and $\sigma^2\geq m/2n$ for sufficiently large $n$. Therefore, writing $X$ for $X_1+\dots+X_{n-z}$, we have that 
\[\P[\sigma\sqrt{n-z}\le X-\E X\le 2\sigma\sqrt{n-z}]\geq\frac 18-\frac\rho{\sigma^3\sqrt{n-z}}\]
and
\[\P[-2\sigma\sqrt{n-z}\le X-\E X\leq -\sigma\sqrt{n-z}]\geq\frac 18-\frac\rho{\sigma^3\sqrt{n-z}}.\]
Note that $\rho/(\sigma^3\sqrt{n-z})\leq 1/40$ 
(since $z\leq 2n/3$ and $m\ge 10^6$), so both these probabilities are at least 1/10.

Since $f_A(j)-g_A(j)=j-1/2$ for every $j\in[n]$, we have 
\[g_A(j)-g_A(j+m)-g_A(j+2m)+g_A(j+3m)=f_A(j)-f_A(j+m)-f_A(j+2m)+f_A(j+3m)\]
for every $j\in[n]$. 

Let $\eta=|\alpha|\sigma\sqrt{n-z}$.
From the above, there is a probability of at least 1/10 that $\a(g_A(j)-g_A(j+m)-g_A(j+2m)+g_A(j+3m))\in[M+\eta,M+2\eta]$ for some $M$ and also a probability of at least 1/10 that $\a(g_A(j)-g_A(j+m)-g_A(j+2m)+g_A(j+3m))\in[M-2\eta,M-\eta]$. 

Since we assumed $|\alpha|\sqrt{m}\le 1/12$, we have that
$\eta\le2|\alpha|\sqrt{m}\le1/6$. It can be easily checked
that for any $M$ and $0\le\eta\le 1/6$, for at least one
of the intervals $[M-2\eta,M-\eta]$ and $[M+\eta,M+2\eta]$,
its distance from integers must be at least
$\eta$.
From this it follows that with probability at least 1/10, the distance from $\a(g_A(j)-g_A(j+m)-g_A(j+2m)+g_A(j+3m))$ to the nearest integer is at least $\eta$, which is at least
$|\alpha|\sqrt{m/6}$.
\end{proof}

Because of the assumptions $|\alpha|\sqrt{m}\le 1/12$ and $m\ge 10^6$, Lemma~\ref{quadruple-clt}
cannot be applied if $|\alpha|>1/(12\cdot 10^3)$. For that case we
shall need an analogous lemma which sets $m=1$ and applies Theorem~\ref{poisson}.

\begin{lemma}\label{quadruple-poisson}
For large enough $n$ the following holds. Let $\alpha$ be such that
$|\alpha|\le 1/2$.
Let $j$ be such that $j+3\le n/2$. Let $z=|\{h:a_h\le j+2\}|$ and assume that
$z\le 2n/3$.

Then, if $A$ is a random sequence in $[n]^n$, after conditioning on 
the values of $z$,
with probability at least $1/10$,
the distance from
$\alpha(g_A(j)-g_A(j+1)-g_A(j+2)+g_A(j+3))$
to the nearest integer
is at least $|\alpha|/4$.
\end{lemma}

\begin{proof}
In a similar way to what we did in the proof of Lemma~\ref{quadruple-clt}, after
conditioning on the value of $z$, we write
$f_A(j+3)=z+(1/2)\sum_{i=1}^{n-z}X_i$, where each of the independent
random variables $X_i$ is Bernoulli with $\E X_i=1/(n-j-2)$.
Let $X=\sum_{i=1}^{n-z}X_i$ and $\lambda=\E X=(n-z)/(n-j-2)$. Due to the assumptions $z\le 2n/3$ and $j+3\le n/2$,
we have that $1/3\le\lambda\le 2$. Then, by Theorem~\ref{poisson},
for large $n$,
\[
\min\{\Pr[X=0],\Pr[X=1]\}\ge\min\{e^{-2},e^{-1/3}/3\}-4/n> 1/10\;.
\]
Therefore, there exists a value $M$ such that 
$\alpha(g_A(j)-g_A(j+1)-g_A(j+2)+g_A(j+3))=M$ with probability at least
$1/10$ and $\alpha(g_A(j)-g_A(j+1)-g_A(j+2)+g_A(j+3))=M+|\alpha|/2$
also with probability at least $1/10$. Since $|\alpha|\le1/2$, 
for at least one of those values
the distance to the nearest integer is at least $|\alpha|/4.$
\end{proof}

The main result of this subsection is the following.

\begin{lemma}\label{fcbound}
For large enough $n$ the following holds.
Let $\alpha$ be such that $|\alpha|\le 1/2$.
Then, with probability at least $1-n^{-10}$ over the choice of a random die $A$, 
for every $\beta$ we have
$|\hat f(\a,\be)|\leq 1-\delta$, where 
$\delta=\min\left(10^{-14},
10^{-12}\frac{\alpha^2n}{\log n}\right)$.
\end{lemma}

\begin{proof}
In fact, we shall prove the statement with probability at least $1-n^{-12}$ for
$A$ chosen as uniform random sequence from $[n]^n$. The statement for a random die
with probability at least $1-n^{-10}$ follows by Lemma~\ref{probzero}.

\medskip

To do this, we first choose some parameters following Lemmas~\ref{quadruple-clt} and~\ref{quadruple-poisson}.
More precisely, if $|\alpha|\ge 1/(12\cdot 10^3)$ let $m=1$
and $\e=|\alpha|/4$,
and otherwise let
$m=\lfloor\min(1/(144\alpha^2),\allowbreak n/(4\cdot 10^6 \log n))\rfloor$
and $\e=|\alpha|\sqrt{m/6}$.
Then pick $k$ maximal such that $4km\leq n/2$ and let $S_1,\dots,S_k$ be consecutive intervals of length $4m$: that is, $S_i=\{4(i-1)m+1,\dots,4im\}$. 
For each $i$, let $E_i$ be the event that at least
$m/20$ out of $m$ indices $j\in\{4(i-1)m+1,\ldots,(4i-3)m\}$
have the property that the distance from
$\alpha(g_A(j)-g_A(j+m)-g_A(j+2m)+g_A(j+3m))$ to the nearest
integer is at least $\e$.
We shall apply 
Lemma~\ref{deviation2}
to show that with high probability at
least $k/30$ of these events occur.

To that end, let $t_s=\{j:a_j\le s\}$. Note that the event $E_i$ depends
only on the values of $t_s$ for $s\le 4im$.
We shall define $D$ as the event that $t_{\lfloor n/2\rfloor}> 5n/8$. This event has probability at most $e^{-n/32}$, by Hoeffding's inequality. Now let $D'_i$ be the event
that $t_{4(i-1)m}>5n/8$. Our objective is to show that
\[
\P[E_i\;\vert\;t_1,\ldots,t_{4(i-1)m},\lnot D'_{i}]\ge 1/24\;.
\]
Once this is proved, it follows by Lemma~\ref{deviation2}
applied to the indicator random variables of $E_1,\ldots,E_k$
that at least $k/30$ events $E_i$ occur except with probability
$e^{-k/28800}+e^{-n/32}$, which is less than $n^{-12}$ since
$k\ge 4\cdot 10^5\log n$.

But if event $E_i$ does occur, then, as discussed, 
by Lemma~\ref{quadruple}
and by triangle inequality
$(4m)^{-1}\big|\sum_{j\in S_i} e(\alpha g_A(j)+\beta j)\big|\le 1-\e^2/80$. If $E_i$ holds for at least $k/30$ values of $i$, 
then since $4km\ge n/2.01$
we have that 
$n^{-1}|\sum_{j=1}^n e(\alpha g_A(j)+\beta j)|\le
1-\e^2/5000$, which is less than $1-\delta$ by a quick case
analysis.

\medskip

It remains to obtain a lower bound for the probability $\P[E_i\;\vert\;t_1,\ldots,t_{4(i-1)m},\lnot D'_i]$. We shall do this as follows. Let $j\in \{4(i-1)m+1,\ldots,(4i-3)m\}$. Let $s=4m(i-1)$ and let
$r=t_s$.
Since $D_i'$ does not hold, 
$r\le 5n/8$. Let $z=t_{j+2m}$. If we condition on
$r$, then the expected value of $z$ is
$r+(j-s+2m)(n-r)/(n-s)\le r+6m$. Since $z\ge r$,
it follows from Markov's inequality that
$z\le r+36m\le 2n/3$ with probability at least $5/6$.

If $m=1$, then the lower bound on the conditional probability
follows from Lemma~\ref{quadruple-poisson}.
In the other case where $m\ge 10^6$, by Lemma~\ref{quadruple-clt},  
the expected number of $j\in\{s+1,\ldots,(4i-3)m\}$
for which
the distance from
$\alpha(g_A(j)-g_A(j+m)-g_A(j+2m)+g_A(j+3m))$
to the nearest integer is at least $|\alpha|\sqrt{m/6}$
is at least $m/10$. If follows that with probability
at least $1/20$ this property holds for at least $m/20$
values of $j$. Overall, the conditional probability
of $E_i$ is at least $5/6\cdot 1/20=1/24$.
\end{proof}

\begin{corollary} \label{smalloutsidebox}
Let $A$ be a random die and let $\hat f$ be the characteristic function of the random variable $(U,V)$, which is uniformly distributed on the set $\{(g_A(j),j-(n+1)/2):j\in[n]\}$. Then, for sufficiently large $n$, with probability at least $1-n^{-5}$ over the choice of $A$ we have the bound
\[|\hat f(\a,\be)|^n\leq n^{-10}\]
for every $(\a,\be)\in[-1/2,1/2]^2$ such that either $|\a|\geq 10^7\log n/n$ or $|\be|\geq 10^9(\log n/n)^{3/2}$.
\end{corollary}

\begin{proof}
We first need to address the fact that
the bound in
Lemma~\ref{fcbound} holds with high probability for a fixed
$\alpha$, whereas we need the bound to hold with high
probability for all $\alpha$. Let $\alpha_i=i/2n^2$
for $i\in\{-n^2,\ldots,n^2\}$. By a union bound, except with probability
$3n^{-8}$ over the choice of a random die, the bound
in Lemma~\ref{fcbound} holds for all values of
$\alpha_i$ and $\beta$. More precisely, for those
values we have that 
$|\hat{f}(\alpha_i,\beta)|
\le 1-\min\left(10^{-14},
10^{-12}\frac{\alpha^2 n}{\log n}\right)$,
which is less than $1-20\log n/n$ for all values of
$i$ such that $|\alpha_i|\ge 10^{7}\log n / n$.
Furthermore, by Lemma~\ref{maxnorm},
with probability at least $1-8n^{-11/2}$ the sampled die is such that  
$\|U\|_\infty\le 2\sqrt{n\log n}$. From now on we fix a die
which satisfies all those bounds. Note that indeed such a die
occurs with probability at least $1-n^{-5}$ for sufficiently large $n$.

Consider another $\alpha$ with absolute value
at least $10^7\log n/n$ and let $i$ be such that 
$|\alpha_i|\ge |\alpha|$ and the distance between $\alpha$
and $\alpha_i$ is minimized. In particular,
$|\alpha_i-\alpha|\le 1/2n^2$. If we show
$|\hat{f}(\alpha_i,\beta)-\hat{f}(\alpha,\beta)|\le 10\log n/n$,
then 
$|\hat{f}(\alpha,\beta)|\le 1-10\log n/n$ and
$|\hat{f}(\alpha,\beta)|^n\le(1-10\log n/n)^n\le n^{-10}$,
as claimed. Indeed, by the definition of characteristic function,
it follows that
\[
|\hat{f}(\alpha_i,\beta)-\hat{f}(\alpha,\beta)|
=\left|\E\left[e(\alpha_iU+\beta V)
\left(1-e((\alpha-\alpha_i)U)\right)
\right]\right|
\le\E|1-e((\alpha-\alpha_i)U)|\;.
\]
Since $|1-e(x)|^2=2(1-\cos2\pi x)\le 4\pi^2x^2$, it indeed follows
that 
$|\hat{f}(\alpha_i,\beta)-\hat{f}(\alpha,\beta)|\le 2\pi|\alpha_i-\alpha|\cdot\|U\|_\infty\le 10\log n/n$.

\medskip

Now suppose that $|\a|\leq 10^7 \log n/n$. 
Suppose first that $|\be|\leq n^{-1}/2$ and let $m=\lfloor n/2\rfloor$. For each $j\leq m$ let us obtain an upper bound for $|e(g_A(j)\a+j\be)+e(g_A(j+m)\a+(j+m)\be)|$. By our bound on $|\a|$, we have that $|g_A(j)\a|$ and $|g_A(j+m)\a|$ are both at most $2\cdot 10^7(\log n)^{3/2}/\sqrt n$. It follows that 
\[\big|g_A(j)\a+j\be-(g_A(j+m)\a+(j+m)\be)\big|\geq n|\be|/3-4\cdot 10^7(\log n)^{3/2}/\sqrt n.\]
Our lower bound for $|\be|$ implies that this is at least $n|\be|/6$. We also have that
\[\big|g_A(j)\a+j\be-(g_A(j+m)\a+(j+m)\be)\big|\leq n|\be|/2+4\cdot 10^7(\log n)^{3/2}/\sqrt n\leq n|\be|\leq 1/2.\]
The proof of Lemma \ref{quadruple} included the inequality that $|e(\theta)+e(\phi)|\leq 2(1-(\theta-\phi)^2)$ when $|\theta-\phi|\leq 1/2$, so using this and the two bounds just noted, we find that 
\[\big|e\big(g_A(j)\a+j\be\big)+e\big(g_A(j+m)\a+(j+m)\be\big)\big|\leq 2(1-n^2\be^2/36),\]
from which it follows that $|\hat f(\a,\be)|\leq 1-n^2\be^2/40$ when $n$ is sufficiently large. (The slight worsening of the absolute constant is to allow for the fact that $n$ may be odd, in which case we do not get an exact partition of $[n]$ into pairs $\{j,j+m\}$.) 

Using our lower bound on $|\be|$ again, we find that $n^2\be^2/40\geq \log^3 n/n$, from which it follows readily that $|\hat f(\a,\be)|^n\leq n^{-10}$.

\medskip

If $1/4\geq|\be|> n^{-1}/2$, then the proof is similar, but this time we choose $m$ maximal such that $m|\be|\leq 1/4$. Note that $m\leq n/2$ if we do this. It follows that we can partition $[n]$ into at least $n/3$ disjoint pairs $\{j,j+m\}$. For each such pair we have that 
\[\big|g_A(j)\a+j\be-(g_A(j+m)\a+(j+m)\be)\big|\geq m|\be|-4\cdot 10^7(\log n)^{3/2}/\sqrt n\geq 1/10\]
and also
\[\big|g_A(j)\a+j\be-(g_A(j+m)\a+(j+m)\be)\big|\leq m|\be|+4\cdot 10^7(\log n
)^{3/2}/\sqrt n\leq 1/3,\]
when $n$ is sufficiently large. It follows that 
\[\big|e(g_A(j)\a+j\be)+e(g_A(j+m)\a+(j+m)\be)\big|\leq 2(1-1/100),\]
and hence that $|\hat f(\a,\be)|\leq 298/300$, which for sufficiently large $n$ gives us that $|\hat f(\a,\be)|^n\leq n^{-10}$.

If $1/4<|\be|\leq 1/2$ then we can argue very similarly but taking $m=1$. We can bound $|\hat f(\a,\be)|$ away from 1 by an even better absolute constant and the required estimate holds with a lot of room to spare. 
\end{proof}

\section{A local central limit theorem for 
\texorpdfstring{$(U,V)$}{(U,V)}
}

At this point, we note that if we choose a random $n$-sided die $A$, then with probability $1-o(1)$ (where the $o(1)$ term is bounded by $2n^{-5}$) we have the conclusions of the main results of the previous section, Lemma \ref{maxnorm} and Corollary \ref{smalloutsidebox}. The first of these gives an upper bound for $\|U\|_\infty$ and the second shows that $|\hat f(\a,\be)|^n$ is small when $(\a,\be)$ lie outside a small box about the origin.

Given these properties, it is now reasonably easy to follow the standard Fourier method to prove a local central limit theorem for the sum of $n$ independent copies of $(U,V)$ that will be strong enough and explicit enough to enable us to prove our main result. So let us fix an $n$-sided die $A$ that has the properties, and let $(U,V)$ be the random variable defined earlier that we associate with $A$.

Let us briefly recall a few standard facts about characteristic functions. One is that if $\hat f$ is the characteristic function of $(U,V)$, then the characteristic function of the sum of $n$ independent copies of $(U,V)$ is $\hat f^n$. This follows from the convolution law in Fourier analysis (since if we regard $(U,V)$ as a function $f$ from $\Z^2$ to $\R$, then $\hat f$ is its Fourier transform, and the function corresponding to the distribution of the sum of $n$ independent copies of $(U,V)$ is the $n$-fold convolution of $f$). For similar reasons we have the inversion formula
\[\P[(U,V)=(x,y)]=\int_{\T^2}\hat f(\a,\be)e(-\a x-\be y)\,d\a\,d\be.\]

We shall also need to know that the characteristic function of $(U,V)$ relates in a simple way to its moments. We have that
\[\frac{\partial^{r+s}}{\partial^r\a\,\partial^s\be}\hat f(\a,\be)=(2\pi i)^{r+s}\E(U^rV^s e(\a U+\be V)),\]
and evaluating this at zero we get $(2\pi i)^{r+s}\E(U^rV^s)$.

We shall use the following estimate, which follows from Taylor's theorem and the observation about the partial derivatives. (We also use the fact that $(U,V)$ has mean $(0,0)$.)

\begin{lemma} \label{taylor}
Let $f$ be as above. Then 
\[\hat f(\a,\be)=1-2\pi^2(\a^2\E U^2+2\a\be\E UV+\be^2\E V^2)+R(\a,\be),\]
where $|R(\a,\be)|\leq\frac{4\pi^3}3(|\a|\cdot\|U\|_\infty+|\be|\cdot\|V\|_\infty)^3$.\hfill$\square$
\end{lemma}

From Lemma \ref{taylor}, we see that for small $\a,\be$, $\hat f(\a,\be)^n$ is approximately equal to $\exp(-nQ(\a,\be))$, where $Q(\a,\be)=2\pi^2(\a^2\E U^2+2\a\be\E UV+\be^2\E V^2)$, which is a positive semidefinite quadratic form in $\a$ and $\be$. Next, we prove two small technical lemmas in order to help us determine sufficient conditions for this approximation to be a good one.

\begin{lemma}\label{expestimate}
For every integer $n\ge 2$ and every pair of real numbers $x,y$ such that $x^2\leq 1/4n$ and $|y|\leq 1/4n$, we have the inequality 
\[\exp(-nx)\exp(-n(|y|+2x^2+2y^2))\leq (1-x+y)^n\leq\exp(-nx)\exp(n(|y|+2x^2+2y^2)).\]
It follows that the ratio of $(1-x+y)^n$ to $\exp(-nx)$ lies between $1-4n(|y|+x^2)$ and $1+4n(|y|+x^2)$.
\end{lemma}

\begin{proof}
For every $u$ with $|u|\leq 1/2$, we have the inequality $-u-u^2\leq \log(1-u)\leq-u+u^2$, which implies that $\exp(-n(u+u^2))\leq(1-u)^n\leq\exp(-n(u-u^2))$. Applying this with $u=x-y$ and noting that $u^2\leq 2(x^2+y^2)$, we obtain the first pair of inequalities.

Also, since $nx^2$ and $n|y|$ are both at most $1/4$, $n(|y|+2x^2+2y^2)\leq
2n(|y|+x^2)\leq 1$. But when $|w|\leq 1$ we have $1-2|w|\leq e^{-w}\leq 1+2|w|$, and the statement about the ratios  follows.
\end{proof}

\begin{lemma}\label{expestimatecomplex}
For every integer $n\ge 2$, every real
number $x$ and every complex number $y$
such that $x^2\le 1/80n$ and $|y|\le 1/80n$, we have
the inequalities
\[
\left|1-\frac{(1-x+y)^n}{\exp(-nx)}\right|,
\left|1-\frac{\exp(-nx)}{(1-x+y)^n}\right|
\le40n(|y|+x^2)\;.
\]
\end{lemma}

\begin{proof}
Let us write $y=\alpha+\beta i$ and let
$L=(1-x+y)^n$ and $z=1-L/\exp(-nx)$. We proceed to estimate
the imaginary part of $z$:
\begin{align*}
\left|\Im L\right|
&=\left|\Im\sum_{k=0}^n\binom{n}{k}(1-x)^{n-k}
\sum_{j=0}^k\binom{k}{j}\alpha^{k-j}(\beta i)^j\right|
\le\sum_{k=1}^n\binom{n}{k}(1-x)^{n-k}
\sum_{j=0}^k\binom{k}{j}|\alpha|^{k-j}|\beta|^j\\
&=(1-x+|\alpha|+|\beta|)^n-(1-x)^n\;,
\end{align*}
Applying Lemma~\ref{expestimate}
to $(1-x+|\alpha|+|\beta|)^n$ and $(1-x)^n$,
we obtain
\[
|\Im z|
=\frac{|\Im L|}{\exp(-nx)}
\le\frac{(1-x+|\alpha|+|\beta|)^n}{\exp(-nx)}
-\frac{(1-x)^n}{\exp(-nx)}
\le 1+4n(2|y|+x^2)-1+4nx^2=8n(|y|+x^2)\;.
\]
By a similar argument, 
the real part of $L$ is bounded by
$\left|\Re L-(1-x)^n\right|
    \le(1-x+|\alpha|+|\beta|)^n-(1-x)^n$
and
\begin{align*}
    \left|\Re z\right|
    &=\left|
    1-\frac{(1-x)^n}{\exp(-nx)}
    + \frac{(1-x)^n-\Re L}{\exp(-nx)}
    \right|
    \nonumber\\
    &\le\left|1-\frac{(1-x)^n}{\exp(-nx)}
    \right|+\frac{(1-x+|\alpha|+|\beta|)^n}{\exp(-nx)}
    -\frac{(1-x)^n}{\exp(-nx)}\le 12n(|y|+x^2)\;.
\end{align*}
Together, the bounds on the real and imaginary parts of $z$ imply
\begin{align*}
    \left|1-\frac{(1-x+y)^n}{\exp(-nx)}\right|
    \le 20n(|y|+x^2)\;.
\end{align*}
Finally, the inverse inequality is obtained with
\begin{align*}
    \left|1-\frac{\exp(-nx)}{(1-x+y)^n}\right|
    &=|z|\frac{\exp(-nx)}{|1-x+y|^n}
    \le\frac{20n(|y|+x^2)}{1-20n(|y|+x^2)}
    \le 40n(|y|+x^2)\;.\qedhere
\end{align*}
\end{proof}

Combining the above lemmas with Corollary \ref{smalloutsidebox}, we obtain the following result, which tells us that in a suitable sense ${\hat f}^n$ is approximated by a Gaussian.

\begin{lemma}\label{gaussianapproximation}
Let $\hat f$ be the characteristic function of $(U,V)$. Define a function $h:\R^2\to\R$ by setting $h(\a,\be)=\hat f(\a,\be)^n$ when $|\a|\leq 1/2$ and $|\be|\leq 1/2$ and $h(\a,\be)=0$ otherwise. And let $g$ be the two-dimensional Gaussian
\[g(\a,\be)=\exp(-2\pi^2n(\a^2\E U^2+2\a\be\E UV+\be^2\E V^2)).\]
Then $\|h-g\|_1\leq 10^{50}(\log n)^7/n^3$.
\end{lemma}

\begin{proof}
Since the conclusion of Lemma \ref{maxnorm}, is satisfied, we have $\|U\|_\infty\leq 2\sqrt{n\log n}$. We also have that $\|V\|_\infty\leq n$ for trivial reasons.

It follows that for any $\a$ and $\be$ we have
\[2\pi^2(\a^2\E U^2+2\a\be\E UV+\be^2\E V^2)\leq 2\pi^2(2\a\sqrt{n\log n}+\be n)^2\leq 160\a^2n\log n+40\be^2n^2.\]
Also, in Lemma \ref{taylor} we have that 
\[|R(\a,\be)|\leq 50(2|\a|\sqrt{n\log n}+|\be|n)^3\leq 200(8|\a|^3(n\log n)^{3/2}+|\be|^3n^3)\leq 2^{11}|\a|^3(n\log n)^{3/2}+2^8|\be|^3n^3.\]
Setting $Q(\a,\be)=2\pi^2(\a^2\E U^2+2\a\be\E UV+\be^2\E V^2)$, so $g(\a,\be)=\exp(-nQ(\a,\be))$, we have that
\[\hat f(\a,\be)=1-Q(\a,\be)+R(\a,\be),\]
and therefore by Lemma \ref{expestimatecomplex} that 
\[
\left|1-\frac{\hat{f}(\alpha,\beta)^n}{g(\alpha,\beta)}\right|,
\left|1-\frac{g(\alpha,\beta)}{\hat{f}(\alpha,\beta)^n}\right|
\le40n(|R(\alpha,\beta)|+Q(\alpha,\beta)^2)\;,
\]
provided that $Q(\a,\be)\leq 1/4\sqrt{5n}$ and $|R(\a,\be)|\leq 1/80n$.

For $Q(\a,\be)$ to be at most $1/4\sqrt 5n$ it is enough if $160\a^2n\log n\leq 1/8\sqrt{5n}$ and $40\be^2n^2\leq 1/8\sqrt {5n}$, and for that it is enough if $|\a|\leq 1/64n^{3/4}\!\!\sqrt{\log n}$ and $|\be|\leq 1/32n^{5/4}$. For $R(\a,\be)$ to be at most $1/80n$ it is enough if $2^{11}|\a|^3(n\log n)^{3/2}\leq 1/160n$ and $2^8|\be|^3n^3\leq 1/160n$, and for that it is enough if $|\a|\leq 1/128 n^{5/6}\!\!\sqrt{\log n}$ and $|\be|\leq 1/64 n^{4/3}$. The second pair of conditions is more stringent, so provided that they hold, we have the required bounds on $Q(\a,\be)$ and $R(\a,\be)$. 

To bound $\|h-g\|_1$, we shall look at the contribution within a small box and outside it. Corollary \ref{smalloutsidebox} tells us that $|\hat f(\a,\be)|^n\leq n^{-10}$ when $|\a|\geq 10^7\log n/n$ or $|\be|\geq 10^9(\log n/n)^{3/2}$. For sufficiently large $n$, $10^7\log n/n\leq 1/200n^{5/6}\sqrt{\log n}$ and $10^9(\log n/n)^{3/2}\leq 1/16n^{4/3}$. It follows that on the boundary of the box $B=[-10^7\log n/n,10^7\log n/n]\times[-10^9(\log n/n)^{3/2},10^9(\log n/n)^{3/2}]$ we have that $|\hat f(\alpha,\beta)|^{n}\le n^{-10}$
and $g(\a,\be)\leq n^{-10}(1+40n(|R(\a,\be)|+Q(\a,\be)^2))$. For sufficiently large $n$ this implies that $g(\a,\be)\leq 20n^{-10}$ everywhere on the boundary. 

We shall now estimate the integral of $g(\a,\be)$ outside the box $B$. Note that for each $(\a,\be)$ on the boundary of $B$, we have that $g(t\a,t\be)$ is a Gaussian in $t$ that takes the value at most $n^{-10}$ when $t=1$. Therefore, it is bounded above by $\exp(-10t^2\log n)$. 
The area of $B$ is at most $4\cdot 10^{16}(\log n/n)^{5/2}$.
Performing the appropriate substitution, we see that
(again assuming that $n$ is sufficiently large)
\[\int_{(\a,\be)\notin B}g(\a,\be)\mathop{d\a}\mathop{d\be}
\le\frac{8\cdot 10^{16}(\log n)^{5/2}}{n^{5/2}}
\int_1^\infty t\exp(-10t^2\log n)\mathop{dt}
\le n^{-12}\;.
\]
Since $h=0$ outside $[-1/2,1/2]^2$ and $|h|\leq n^{-10}$ outside $B$, we also have that 
\[\int_{(\a,\be)\notin B}|h(\a,\be)|\mathop{d\a}\mathop{d\be}\leq n^{-10}.\]
For all $(\a,\be)$ we have that $|f(\a,\be)^n|\leq 1$. Therefore, our estimate for the ratio of $g$ to $f$ implies that for all $(\a,\be)\in B$ we have that 
\[|g(\a,\be)-\hat f(\a,\be)^n|\leq 40n(|R(\a,\be)|+Q(\a,\be)^2).\]
Recall that we have previously estimated
$Q(\alpha,\beta)\le 160\alpha^2 n\log n+40\beta^2n^2$
and
$|R(\alpha,\beta)|\le 2^{11}|\alpha|^3(n \log n)^{3/2}
+2^8|\beta|^3n^3$.
Applying these bounds inside $B$ one can check that
$Q(\a,\be)\leq 10^{20}(\log n)^3/n$ for $n$ sufficiently large, and hence that $Q(\a,\be)^2\leq 10^{40}(\log n)^6/n^2$. Also, $|R(\a,\be)|\leq 10^{30}(\log n)^{9/2}/n^{3/2}$. So, for sufficiently large $n$, we have the bound
\[|g(\a,\be)-\hat f(\a,\be)^n|\leq 10^{32}(\log n)^{9/2}/\sqrt{n}\]
everywhere in $B$. Since $B$ has area at most $4\cdot 10^{16}(\log n/n)^{5/2}$, it follows that the contribution to $\|g-h\|_1$ from inside $B$ is at most $10^{49}(\log n)^7/n^3$.

Combining these estimates gives us the bound stated (as always, assuming that $n$ is sufficiently large).
\end{proof}



In the statement of the next result, 
we shall write $(U^{*n}, V^{*n})$ for
$(\sum_{i=1}^n U_i,\sum_{i=1}^n V_i)$, where
each $(U_i,V_i)$ is independently and identically
distributed according to the distribution of
$(U,V)$.
We also refer to a ``discrete Gaussian". By this we mean a 
function defined on $\mathbb Z^2$ with a formula of the form $f(x,y)=c\exp(-\lambda q(x,y))$ for some positive semidefinite quadratic form $q$.

\begin{corollary} \label{discretegaussian}
There is a discrete Gaussian $G$ such that
\[\big|\P\big[(U^{*n},V^{*n})=(x,y)\big]-G(x,y)\big|\leq 10^{50}(\log n)^7/n^3\]
for every $(x,y)\in\mathbb Z^2$. Furthermore, $\P[(U^{*n},V^{*n})=(0,0)]\leq 10^{17}(\log n/n)^{5/2}$.
\end{corollary}

\begin{proof}
We have
\[\P[(U^{*n},V^{*n})=(x,y)]=\int_{\mathbb T^2}\hat f(\a,\be)^ne(-\a x-\be y)\mathop{d\a}\mathop{d\be}=\int_{\R^2}h(\a,\be)e(-\a x-\be y)\mathop{d\a}\mathop{d\be}.\]
By Lemma \ref{gaussianapproximation} and the fact that $e(-\a x-\be y)$ has modulus 1 for every $x,y$, if we replace $h$ by $g$ on the right-hand side, the difference to the integral is at most $10^{50}(\log n)^7/n^3$. But $\int_{\R^2}g(\a,\be)e(-\a x-\be y)\mathop{d\a}\mathop{d\be}$ has a Gaussian dependence on $(x,y)$, since it is the Fourier transform of a Gaussian (the Gaussian being defined on $\R^2$ even if we are evaluating its Fourier transform at points of $\Z^2$). This proves the first part.

For the second part, recall from the proof of Lemma \ref{gaussianapproximation} that the integral of $h$ outside the box $B$ is at most $n^{-10}$ and that $B$ has area at most $4\cdot 10^{16}(\log n/n)^{5/2}$. It follows that \[\P[(U^{*n},V^{*n})=(0,0)]=\int_{\R^2}h(\a,\be)\mathop{d\a}\mathop{d\be}\leq 10^{17}(\log n/n)^{5/2}.\] 
when $n$ is sufficiently large. (Indeed, this is a bound for $\P[(U^{*n},V^{*n})=(x,y)]$ for all $(x,y)$.)
\end{proof}

\section{The main theorems}

As before, we consider a fixed die $A$ that satisfies the conclusions
of Lemma~\ref{maxnorm} and Corollary~\ref{smalloutsidebox}.
Recall that $\P[V^{*n}=0]\ge n^{-3/2}/4$, by Lemma~\ref{probzero}.
Hence, by Corollary~\ref{discretegaussian},
the probability that $A$ ties with a random die $B$,
which is equal to $\P[U^{*n}=0\;\vert\;V^{*n}=0]$, is at most
$10^{18}(\log n)^{5/2}/n$. Since this conclusion holds for $1-o(1)$
fraction of dice $A$, we have directly established Theorem~\ref{ties-thm}.

We are also almost ready to prove Theorem~\ref{transitive-thm}. However, a uniform bound on the probabilities is not quite enough for our purposes. As is customary, we need to combine it with tail estimates. However, this is straightforward.

\begin{lemma}\label{tail}
$\P[|U^{*n}|\geq 2Cn\sqrt{\log n}]\leq 2\exp(-C^2/2)$. 
\end{lemma}

\begin{proof}
The distribution of $U^{*n}$ is given by the sum of $n$ independent copies of $U$, which has mean zero. Since $\|U\|_\infty\leq 2\sqrt{n\log n}$, Hoeffding's inequality gives us the required estimate.
\end{proof}

\begin{theorem}
The probability that $A$ beats another random die is $\frac 12+o(1)$.
\end{theorem}

\begin{proof}
As remarked earlier, we will be done if we can prove that 
\[\P[U^{*n}>0\mid V^{*n}=0]=\frac 12+o(1).\]
First, recall again that $\P[V^{*n}=0]\geq n^{-3/2}/4$.

Next, note that by Corollary \ref{discretegaussian} we have for every $x$ that 
\[\big|\P\big[(U^{*n},V^{*n})=(x,0)\big]-\P\big[(U^{*n},V^{*n})=(-x,0)\big]\big|\leq 2\cdot 10^{50}(\log n)^7/n^3,\]
since $G$ is an even function. 
Thirdly, Lemma \ref{tail} gives us that $\P[|U^{*n}|>8 n\log n]\leq 2n^{-8}$.

Putting these estimates together, we find that
\begin{align*}
\big|\P\big[U^{*n}>0\wedge V^{*n}=0\big]-\P\big[U^{*n}<0\wedge V^{*n}=0\big]\big|
&\leq(8n\log n)(2\cdot 10^{50}(\log n)^7/n^3)+2n^{-8}\\
&\leq 10^{52}(\log n)^8/n^2,
\end{align*}
and therefore that
\[\big|\P\big[U^{*n}>0|V^{*n}=0\big]-\P\big[U^{*n}<0\mid V^{*n}=0\big]\big|\leq 10^{52}(\log n)^8/\sqrt{n}.\]
As discussed before, we also have that 
\[\P[U^{*n}=0 \mid V^{*n}=0]\leq 10^{17}(\log n)^{5/2}/n.\]
The result follows.
\end{proof}

As observed at the end of 
Section \ref{rvs}, since this statement 
holds for $1-o(1)$ fraction of dice $A$,
Theorems~\ref{ties-thm} and~\ref{transitive-thm}
follow from it.

\bibliographystyle{alpha}
\bibliography{biblio_arxiv_v2}

\newcommand{\etalchar}[1]{$^{#1}$}
\begin{thebibliography}{CGG{\etalchar{+}}16}

\bibitem[BH84]{poisson}
Andrew~D. Barbour and Peter Hall.
\newblock On the rate of {P}oisson convergence.
\newblock {\em Mathematical Proceedings of the Cambridge Philosophical
  Society}, 95(3):473--480, 1984.

\bibitem[CG91]{CG}
F.R.K. Chung and R.L. Graham.
\newblock Quasi-random tournaments.
\newblock {\em Journal of Graph Theory}, 15(2):173--198, 1991.

\bibitem[CGG{\etalchar{+}}16]{CGG}
Brian Conrey, James Gabbard, Katie Grant, Andrew Liu, and Kent~E. Morrison.
\newblock Intransitive dice.
\newblock {\em Mathematics Magazine}, 89(2):133--143, 2016.

\bibitem[CH20]{fourdice}
Elisabetta Cornacchia and Jan Hązła.
\newblock Intransitive dice tournament is not quasirandom.
\newblock arXiv:2011.10067, 2020.

\bibitem[GW17]{LCLT}
Rita Giuliano and Michel Weber.
\newblock {Approximate local limit theorems with effective rate and application
  to random walks in random scenery}.
\newblock {\em Bernoulli}, 23(4B):3268--3310, 2017.

\bibitem[She14]{shevtsova}
I.G. Shevtsova.
\newblock On the absolute constants in the {Berry--Esseen-type} inequalities.
\newblock {\em Doklady Mathematics}, 89:378--381, 2014.

\end{thebibliography}
\end{document}